\documentclass[letterpaper, 10 pt, conference]{ieeeconf}  

\usepackage{amsmath}
\usepackage{amssymb}
\usepackage{mathrsfs}
\usepackage{bbm}
\usepackage{mathtools}
\usepackage{algorithm}
\usepackage{graphicx}
\usepackage{units}
\usepackage{tikz}

\usepackage{xcolor}

\usepackage{xifthen}

\renewcommand{\t}{^{\mbox{\tiny \sf T}}}
\newcommand{\m}{^{\mbox{\scriptsize{-}}}}

\newcommand{\inv}{^{- 1}}

\DeclareMathOperator{\E}{E}
\DeclareMathOperator{\tr}{tr}

\DeclareMathOperator{\diag}{diag}

\DeclareMathOperator{\ncdf}{cdfn}

\let\Pr\relax
\newcommand{\Pr}{\mathbb{P}}

\renewcommand{\Re}{\mathbb{R}}

\renewcommand{\E}{\mathbb{E}}

\newcommand{\F}{\mathscr{F}}
\newcommand{\Ncal}{\mathcal{N}}

\newcommand{\isomorph}{\cong}

\renewcommand{\Re}{\mathbb{R}}

\DeclareSymbolFont{matha}{OML}{txmi}{m}{it}
\DeclareMathSymbol{\varv}{\mathord}{matha}{118} 
\DeclareMathSymbol{\varw}{\mathord}{matha}{119} 

\newtheorem{problem}{Problem}
\newtheorem{remark}{Remark}
\newtheorem{theorem}{Theorem}[section]

\graphicspath{ {./img/} }

\IEEEoverridecommandlockouts                              

\renewcommand{\baselinestretch}{0.99}

\title{\LARGE \bf
Chance Constrained Covariance Control for Linear\\Stochastic Systems With Output Feedback
}

\author{Jack Ridderhof \and Kazuhide Okamoto \and Panagiotis Tsiotras%
\thanks{J. Ridderhof is a PhD student with the School of Aerospace Engineering, Georgia Institute of Technology, Atlanta, GA, 30332-0150, USA. Email: jridderhof3@gatech.edu}%
\thanks{K. Okamoto is a PhD student with the School of Aerospace Engineering, Georgia Institute of Technology, Atlanta, GA, 30332-0150 USA. Email: kazuhide@gatech.edu}%
\thanks{P. Tsiotras is the Andrew and Lewis Chair Professor with the D. Guggenheim School of Aerospace Engineering, and the Institute for Robotics and Intelligent Machines, Georgia Institute of Technology, Atlanta, GA 30332-0150, USA. Email: tsiotras@gatech.edu}
}

\begin{document}

\maketitle
\thispagestyle{empty}
\pagestyle{empty}

\begin{abstract}

We consider the problem of steering, via output feedback, the state distribution of a discrete-time, linear stochastic system from an initial Gaussian distribution to a terminal Gaussian distribution with prescribed mean and maximum covariance, subject to probabilistic path constraints on the state.
The filtered state is obtained via a Kalman filter, and the problem is formulated as a deterministic convex program in terms of the distribution of the filtered state.
We observe that, in the presence of constraints on the state covariance, and in contrast to classical Linear Quadratic Gaussian (LQG) control, the optimal feedback control depends on both the process noise and the observation model.
The effectiveness of the proposed approach is verified using a numerical example. 

\end{abstract}

\section{INTRODUCTION}
The problem of covariance steering is a stochastic optimal control problem aiming to design a controller that steers the state covariance of a stochastic system to a target terminal value, while minimizing the expectation of a quadratic function of the state and the input. 
In this work, we focus on a discrete-time linear time-varying stochastic system with partially observable state, a given Gaussian initial state distribution, and an independent and identically distributed (i.i.d.) standard Gaussian additive diffusion to the dynamics and the measurement. 

The infinite horizon covariance control problem for linear time invariant systems has been researched since the late 80's.
In~\cite{hotz1985covariance,hotz1987covariance} the authors investigated the state-feedback gains that \emph{assign} a state covariance value to the system, i.e., the system state covariance converges asymptotically to the assigned value.
%
The finite horizon case has only recently gained attention~\cite{Chen2016a,bakolas2016optimalCDC,Bakolas2018,beghi1996relative,Goldshtein2017,halder2016finite}. 
Chance constraints, which are probabilistic constraints that impose a maximum probability of constraint violation, were first introduced to the covariance control problem in~\cite{Okamoto2018}.
The latter work draws connections between covariance control and a large class of stochastic control problems for which chance constraints are utilized in order to guarantee performance under uncertainty~\cite{blackmore2009convex, geletu2013advances}, such as stochastic model predictive control (SMPC)~\cite{farina2016stochastic,mesbah2016stochastic} and vehicle path planning in belief space~\cite{blackmore2011chance,Okamoto2018b}.

The majority of covariance control research only considers controlling the state covariance.
However, as discussed in~\cite{Okamoto2018}, when state chance constraints exist, the mean and the state covariance are coupled.
The authors in~\cite{Okamoto2018} introduced the first covariance steering controller that simultaneously deals with the mean and the covariance dynamics such that the resulting trajectories satisfy the state chance constraints. 
The approach was further modified to be computationally more efficient in~\cite{Okamoto2018b}, which was eventually extended to deal with input hard constraints~\cite{okamoto2019input} and nonlinear dynamics~\cite{ridderhof2019nonlinear}.
Furthermore, covariance control theory was applied to autonomous vehicle control in~\cite{chohan2019robust} and spacecraft control in~\cite{ridderhof2018uncertainty,ridderhof2019minimumFuel,ridderhof2020lowthrust}.
Finally, in~\cite{okamoto2019stochastic}, the authors applied covariance control theory to SMPC for linear time-invariant systems under unbounded additive disturbance.

The above-mentioned research on covariance control assumes full state feedback. 
This paper, in contrast, is concerned with covariance steering for the case when the state is only indirectly accessible via noisy measurements.

The problem of output-feedback covariance steering has been visited in~\cite{Chen2015,bakolas2017covariance}, where the problem had no constraints other than a terminal boundary constraint. 
Thus, these works only dealt with the control of the state covariance and did not consider mean dynamics. 
The proposed approach deals with state chance constraints 
and simultaneously steers the mean and the covariance of the system state, and thus, can be applied to more realistic scenarios. 

A similar problem setup as the one addressed in this paper has also been visited from the SMPC community~\cite{farina2015approach}, where an output feedback controller was designed to deal with chance constraints. 
Although the approach in~\cite{farina2015approach} successfully computes control commands that satisfy all the constraints, the control policy suffers from conservativeness due to the convex relaxation of the covariance dynamics. 
The approach in our work extends the covariance steering controller in~\cite{Okamoto2018b}, which allows a direct assessment of the covariance at each time step and eliminates the need to conduct conservative convex relaxations of the covariance dynamics. 

The main contribution of this work is the development of a novel covariance steering control policy for linear systems with Gaussian process and measurement noise. 
The proposed approach is a nontrivial extension of the full-state feedback covariance control policy proposed in~\cite{Okamoto2018b}, which allows one to directly assess the value of the covariance at each time step, while converting the original stochastic control problem to a deterministic convex programming problem. 
We observe that, as a direct consequence of the constraints on the state covariance, and in contrast to the classical LQG solution~\cite{Tse1971}, the optimal feedback control depends on both the process noise and the observation model.

\subsection*{Notation}

For a sequence $x = (x_1, \ldots, x_N) = (x_k)_{k = 1}^N$, we use the shorthand $(x_k)$ to refer to the sequence and write $x_k$ to refer to an element of that sequence. We write $\sigma(z)$ to denote the $\sigma$-algebra generated by the random variable $z$. For a symmetric matrix $A$, we write $A > 0$ $(\geq 0)$ if $A$ is positive (semi-)definite.

\section{PROBLEM STATEMENT}\label{sec:ProblemStatement}

Consider the stochastic discrete-time linear system given by
\begin{equation} \label{eq:state_process}
	x_{k + 1} = A_k x_k + B_k u_k + G_k w_k,
\end{equation}
for $k = 0, 1, \dots, N - 1$, where $x_k \in \Re ^ {n_x}$ and $u_k \in \Re ^ {n_u}$ are the state and control, and $A_k \in \Re^{n_x\times n_x}$, $B_k \in \Re^{n_x \times n_u}$, and $G_k \in \Re^{n_x\times n_x}$ are system matrices.
Increments of the disturbance process $w_k \in \Re ^{n_x}$ are i.i.d. standard Gaussian random vectors. The state is measured through the observation process
\begin{equation} \label{eq:observation_process}
	y_k = C_k x_k + D_k v_k,
\end{equation}
where $y_k \in \Re ^{n_y}$ is the measurement and  $v_k \in \Re ^{n_y}$ is measurement noise, and $C_k \in \Re^{n_y \times n_x}$ and $D_k \in \Re^{n_y \times n_y}$ are given. 
Increments of the measurement noise $v_k$ are i.i.d. standard Gaussian random vectors.
Also, and in order to simplify the filtering equations, we assume that 
the matrix $D_k$ is invertible.
The case when $D_k$ is rank-deficient can be treated using well-known approaches~\cite{FairmanLukKF1985}.
Before the first measurement is taken, we assume that we will be provided with a state estimate $\hat{x}_{0\m}$ with estimation error $\tilde{x}_{0\m} = x_0 - \hat{x}_{0\m}$.
We assume that $\hat{x}_{0\m}$ and $\tilde{x}_{0\m}$ are independent random vectors with known distributions given as
\begin{equation}
	\hat{x}_{0\m} \sim \Ncal(\bar{x}_0, \hat{P}_{0\m}), \quad \tilde{x}_{0\m} \sim \Ncal(0, \tilde{P}_{0\m}),
\end{equation}
where the positive semi-definite matrices $\tilde{P}_{0\m}$, $\hat{P}_{0\m}$ and the vector $\bar{x}_0$ are all fixed and known.
That is, we do not assume to know the initial state estimate when designing the control law, but we know its distribution.
This allows for the control law to be designed before all measurements are collected.
For example, in the case when we will be provided  with the exact value of the state, then $\tilde{x}_{0\m} = 0$, \mbox{$\hat{x}_{0\m} = x_0$}, and $\tilde{P}_{0\m} = 0$.
On the other hand, if we will not be provided with any new information about the state before step $k = 0$, then $\hat{x}_{0\m} = \bar{x}_0$ and $\hat{P}_{0\m} = 0$.
Finally, we assume that $\hat{x}_{0\m}$, $\tilde{x}_{0\m}$, $(w_k)$, and $(v_k)$ are independent.

Define the filtration $(\F_k)_{k = -1}^N$ by $\F_{- 1} = \sigma(\hat{x}_{0\m})$ and $\F_k = \sigma(\hat{x}_{0\m}, y_i : 0 \leq i \leq k)$ for $0 \leq k \leq N$.
This filtration represents the information that can be used to estimate the state and determine the control action, in the sense that the estimated state and the control at step $k$ are both $\F_k$-measurable random vectors.
The initial $\sigma$-algebra $\F_{- 1}$ is defined
for logical consistency, since the initial state estimate is known before any measurements are taken.

Let $\bar{x}_k = \E(x_k)$ be the mean state, and define the estimated (filtered) state as $\hat{x}_k = \E(x_k \vert \F_k)$ and the estimation error as $\tilde{x}_k = x_k - \hat{x}_k$. The estimated state has mean
\begin{equation}
	\E(\hat{x}_k) = \E( \E(x_k \vert \F_k) ) = \E (x_k) = \bar{x}_k,
\end{equation}
and hence the estimation error has zero mean, that is, $\E(\tilde{x}_k) = 0$ for all $0 \le k \le N$.
Define the state, estimated state, and estimation error covariances as
\begin{align}
       {P}_k &= \E [({x}_k - \bar{x}_k)({x}_k - \bar{x}_k) \t], \\
	\hat{P}_k &= \E [(\hat{x}_k - \bar{x}_k)(\hat{x}_k - \bar{x}_k) \t], \\
	 \tilde{P}_k &=  \E(\tilde{x}_k \tilde{x}_k \t ) =
	 \E [(\hat{x}_k - x_k)(\hat{x}_k - x_k) \t].
\end{align}
The estimated state is uncorrelated with the estimation error, since
\begin{align}
	\E(\hat{x}_k \tilde{x}_k\t) &= \E [\hat{x}_k (x_k - \hat{x}_k)\t] \nonumber\\
	&= \E [ \E [\hat{x}_k (x_k - \hat{x}_k)\t \vert \F_k]] \nonumber\\
	&= \E [ \hat{x}_k  \E [x_k \t - \hat{x}_k \t \vert \F_k]] \nonumber\\
	&= \E [ \hat{x}_k  (\E [x_k \t \vert \F_k]  - \hat{x}_k \t)] = 0,
\end{align}
and from this expression it can be shown that the state covariance satisfies $P_k = \hat{P}_k + \tilde{P}_k$. 
Define the prior estimated state and prior estimation error as $\hat{x}_{k\m} = \E(x_k \vert \F_{k - 1})$ and $\tilde{x}_{k\m} = x_k - \hat{x}_{k\m}$, respectively,
with corresponding covariances $\hat{P}_{k\m}$ and $\tilde{P}_{k\m}$ as above.
It follows that the initial state is distributed as
\begin{equation} \label{eq:initial_state_dist}
	x_0 \sim \Ncal ( \bar{x}_0, P_0),
\end{equation}
where $P_0 = \hat{P}_{0\m} + \tilde{P}_{0\m}$.
We require that
\begin{equation} \label{eq:chance_constraint}
	\Pr (x_k \notin \chi) \leq p_{\text{fail}}, \quad k = 0, 1, \dots, N,
\end{equation}
where $0 < p_{\text{fail}} < 0.5$ is fixed, and where
\begin{equation} \label{eq:chance_constrained_region}
	\chi = \bigcap_{j = 1}^{N_\chi} \{ x : \alpha_j \t x \leq \beta_j \} \subset \Re ^{n_x},
\end{equation}
where $\alpha_j \in \Re^{n_x}$ and $\beta_j \in \Re$.
Here the compliment of $\chi$ denotes a forbidden region in the state space, and thus we 
we constrain the probability that the state is not in $\chi$ to be no more \mbox{than $p_{\text{fail}}$}~\cite{blackmore2009convex,blackmore2011chance}.
Constraints of the form (\ref{eq:chance_constraint}) are often referred to as \textit{chance constraints}, and, likewise, optimization subject these constraints is referred to as chance constrained optimization~\cite{mesbah2016stochastic}.
%


Finally, we assume for the remainder of this paper that the control input $u_k$ at each step is an affine function of the measurement data.
We say that a control sequence $(u_k)$ is \textit{admissible} if it satisfies this property at every step.
This assumption is made to ensure that if $x_k$ is Gaussian, then the state $x_{k + 1}$ will also be Gaussian.
It follows that, since the state is initially Gaussian distributed, the state will be Gaussian distributed over the entire problem horizon, even in the presence of the chance constraints.


This paper is concerned with the following stochastic optimal control problem.

\begin{problem} \label{prob:original}
	Find the admissible control sequence $u = (u_k)_{k=0}^{N-1}$ such that the chance constraints (\ref{eq:chance_constraint}) are satisfied; the state at step $N$ is distributed according to $\Ncal (\bar{x}_f, P_N)$ for $P_N = \hat{P}_N + \tilde{P}_N \leq P_f$, where $\bar{x}_f$ and $P_f$ are given; and minimizes the cost functional
	\begin{equation} \label{eq:orig_objective}
		J(u) = \E \left( \sum_{k = 0}^{N - 1} x_k \t Q_k x_k + u_k \t R_k u_k \right), 
	\end{equation}
	for a given sequences of matrices $(Q_k \geq 0)$ and $(R_k > 0)$.
\end{problem}


\begin{remark}
  For simplicity, we do not consider chance constraints on the control. However, the method developed in this work may be easily extended to include chance constraints on the control. See, for instance,~\cite{ridderhof2019minimumFuel,okamoto2019stochastic}.
\end{remark}

\subsection{Separation of the Observation and Control Problems}

Since the system is linear and the state is Gaussian distributed, the estimated state may be obtained by the Kalman filter. 
That is, the filtered state satisfies~\cite{BrysonHo}
\begin{equation} \label{eq:kf_filtered_state_process}
	\hat{x}_k = \hat{x}_{k\m} + L_k (y_k - C_k \hat{x}_{k\m}),
\end{equation}
\begin{equation} \label{eq:kf_prior_filtered_state_process}
	\hat{x}_{k\m} = A_{k - 1} \hat{x}_{k - 1} + B_{k - 1} u_{k - 1},
\end{equation}
where
\begin{equation}
	L_k = \tilde{P}_{k\m} C_k \t (C_k \tilde{P}_{k\m} C_k \t + D_k D_k \t) \inv
\end{equation}
is the Kalman gain, and the error covariances are given by
\begin{equation}
	\tilde{P}_k = (I - L_k C_k) \tilde{P}_{k\m} (I - L_k C_k) \t + L_k D_k D_k \t L_k \t,
\end{equation}
\begin{equation}
	\tilde{P}_{k\m} = A_{k - 1} \tilde{P}_{k - 1} A_{k - 1} \t + G_{k - 1} G_{k - 1} \t.
\end{equation}
We see that the estimation error covariance $\tilde{P}_{k}$ does not depend on the control. Using properties of conditional expectation, it is easy to show that
\begin{equation}
	\E (x_k \t Q_k x_k) = \tr \tilde{P}_k Q_k + \E( \hat{x}_k \t Q_k \hat{x}_k),
\end{equation}
and therefore the objective may be rewritten as
\begin{equation} \label{eq:objective_separation}
	J(u) = \sum_{k = 0}^{N - 1} \tr \tilde{P}_k Q_k + \hat{J}(u),
\end{equation}
where
\begin{equation} \label{eq:filtered_state_objective}
	\hat{J}(u) = \E \left( \sum_{k = 0}^{N - 1} \hat{x}_k \t Q_k \hat{x}_k + u_k \t R_k u_k \right).
\end{equation}
Since the estimation error covariance $\tilde{P}_k$ is determined by the Kalman filter and not by the control, optimizing over the objective $\hat{J}(u)$ is equivalent to optimizing over $J(u)$.
Furthermore, we can determine the distribution of the state as a function of the mean and covariance of the \textit{estimated} state process, that is,
\begin{equation} \label{eq:state_dist_equivalence}
	x_k \sim \Ncal (\bar{x}_k, P_k) \iff 
	\hat{x}_k \sim \Ncal (\bar{x}_k, P_k - \tilde{P}_k).
\end{equation}
It follows that, in order for the final state covariance to satisfy $0 < P_N \leq P_f$, the maximum final covariance $P_f$ must satisfy $P_f > \tilde{P}_N$.
Define now the \textit{innovation process} $(\tilde{y}_{k\m})$ by
\begin{align} \label{eq:innovation_process}
	\tilde{y}_{k\m} &= y_k - \E (y_k \vert \F_{k - 1}),
\end{align}
%
for $k = 0, 1, \dots, N$. Since
\begin{equation}
	\E (y_k \vert \F_{k - 1}) = \E (C_k x_k + D_k v_k \vert \F_{k - 1}) = C_k \hat{x}_{k\m},
\end{equation}
we obtain, by substituting the observation model (\ref{eq:observation_process}) in (\ref{eq:innovation_process}), that
\begin{equation} \label{eq:innovation_process_algebraic}
		\tilde{y}_{k\m}  = y_k - C_k \hat{x}_{k\m} = C_k \tilde{x}_{k\m} + D_k v_k.
\end{equation}
%
%
The state error $\tilde{x}_{k\m}$ depends linearly on $\tilde{x}_{0\m}$, $(w_i)_{i = 1}^{k - 1}$, and $(v_i)_{i = 1}^{k - 1}$, which are each independent of $v_k$.
It follows that $\tilde{x}_{k\m}$ and $v_k$ are independent, and therefore
%
%
we can compute the covariance of the innovation process as
\begin{equation} \label{eq:innovation_covariance}
	P_{\tilde{y}_{k\m}} = \E (\tilde{y}_{k\m} \tilde{y}_{k\m} \t ) = C_k \tilde{P}_{k\m} C_k \t + D_k D_k \t.
\end{equation}
Thus, the distribution of the innovation process is determined by the estimation error covariance $\tilde{P}_{k\m}$, and therefore may be computed prior to solving for the control inputs. We rewrite the estimated state process as
\begin{equation} \label{eq:xhat_process}
	\hat{x}_{k + 1} = A_k \hat{x}_k + B_k u_k + L_{k + 1} \tilde{y}_{(k + 1) \m},
\end{equation}
where $\hat{x}_0 = \hat{x}_{0\m} + L_0 \tilde{y}_{0\m}$. We have thus replaced the state process (\ref{eq:state_process}) with noise term $G_k w_k$ with a corresponding filtered state process with noise $L_{k + 1} \tilde{y}_{(k + 1) \m}$. The stochastic optimal control problem may now be posed entirely in terms of the filtered state process~(\ref{eq:xhat_process}).

\subsection{Block-Matrix Formulation}

The filtered state process (\ref{eq:xhat_process}) may be written in matrix notation as
\begin{multline} \label{eq:filtered_state_matrix_form}
	\begin{bmatrix}
		\hat{x}_0 \\ \hat{x}_1 \\ \hat{x}_2 \\ \vdots
	\end{bmatrix}
	= \begin{bmatrix}
		I \\ A_0 \\ A_1 A_0 \\ \vdots
	\end{bmatrix} \hat{x}_{0\m}
	+ \begin{bmatrix}
		0 & 0 &  \\
		B_0 & 0 &  \\
		A_1 B_0 & B_1 &  \\
		&& \ddots
	\end{bmatrix} \begin{bmatrix}
		u_0 \\ u_1 \\ \vdots
	\end{bmatrix} \\
+ \begin{bmatrix}
		L_0 & 0 & 0 &  \\
		A_0 L_0 & L_1 & 0 & \\
		A_1 A_0 L_0 & A_1 L_1 & L_2 \\
		&&& \ddots
	\end{bmatrix} \begin{bmatrix}
		\tilde{y}_{0\m} \\ \tilde{y}_{1\m} \\ \tilde{y}_{2\m} \\ \vdots
	\end{bmatrix}.
\end{multline}
Let $\hat{X}$ and $\tilde{Y}$ be column vectors constructed by stacking $\hat{x}_k$ and $\tilde{y}_{k\m}$ for $k = 0, 1, \dots, N$, and, similarly, let $U$ be the column vector constructed by stacking $u_k$ for $k = 0, 1, \dots, N - 1$.
Formally, we have that the column vector $\hat{X}$ is isomorphic to the sequence $(\hat{x}_k)$, which we denote by $(\hat{x}_k) \isomorph \hat{X}$ (similarly, $(u_k) \isomorph U$, $(\tilde{y}_{k\m}) \isomorph \tilde{Y}$). For appropriately constructed block matrices $A$, $B$, and $L$ as in (\ref{eq:filtered_state_matrix_form}), the filtered state process can be written as the linear matrix equation
\begin{equation} \label{eq:filtered_process_vector_def}
	\hat{X} = A \hat{x}_{0\m} + B U + L \tilde{Y}.
\end{equation}
See \cite{Okamoto2018b,Okamoto2018} for details on this construction. We may then rewrite the cost function (\ref{eq:filtered_state_objective}) in matrix form as
\begin{equation} \label{eq:objective_matrix_form}
	\hat{J}(U) = \E(\hat{X} \t Q \hat{X} + U \t R U),
\end{equation}
where $Q = \mathrm{blkdiag}(Q_0, \dots, Q_{N - 1}, 0) \geq 0$ and $R = \mathrm{blkdiag}(R_0, \dots, R_{N - 1}) > 0$. Let $E_k$ be a matrix defined so that $E_k \hat{X} = \hat{x}_k$, and denote the mean state by
$\bar{X} = \E(\hat{X}) \isomorph (\bar{x}_k)$.
We can then write the terminal state distribution constraints as
\begin{subequations} \label{eq:terminal_matrix_constraints}
	\begin{equation}
		E_N \bar{X} = \bar{x}_f,
	\end{equation}
	\begin{equation} \label{eq:problem2TerminalCovConst}
		E_N \E[(\hat{X} - \bar{X}) (\hat{X} - \bar{X} )\t] E_N \t \leq P_f - \tilde{P}_N.
	\end{equation}
\end{subequations}
Finally, in terms of the column vector $\tilde{X} \isomorph (\tilde{x}_k)$, the chance constraints (\ref{eq:chance_constraint}) may be written as
\begin{equation} \label{eq:chance_constraint_matrix_form}
	\Pr ( E_k(\hat{X} + \tilde{X}) \notin \chi ) \leq p_{\text{fail}}, \quad k = 0, 1, \dots, N.
\end{equation}
The distribution of $\hat{X} + \tilde{X}$ is determined, per (\ref{eq:state_dist_equivalence}), by the filtered process (\ref{eq:filtered_process_vector_def}) and the sequence $(\tilde{P}_k)$, and therefore the probability in (\ref{eq:chance_constraint_matrix_form}) depends solely, for fixed problem parameters, on the control sequence $U$.
In summary, we have reformulated the original stochastic optimal control problem (\ref{prob:original}) in terms of the inaccessible state into the following problem in terms of the accessible filtered state.

\begin{problem} \label{prob:matrix_filtered_problem_affine}
	Find the admissible control sequence $U^* \isomorph (u^*_k)$ that minimizes the objective (\ref{eq:objective_matrix_form}) subject to the terminal constraints (\ref{eq:terminal_matrix_constraints}), for $P_f > \tilde{P}_N$, and chance constraints (\ref{eq:chance_constraint_matrix_form}).
\end{problem}

\section{CONTROL OF THE FILTERED STATE}\label{sec:ControlOfFilteredState}

We consider filtered state history feedback of the form
\begin{equation} \label{eq:control_structure}
	u_k = \sum_{i = 0}^{k } K_{k,i} (\hat{x}_i - \bar{x}_i) + m_k, 
\end{equation}
for $k = 0, 1, \dots, N - 1$, where $K_{k, i} \in \Re ^{n_u \times n_x}$ are feedback gains and $m_k \in \Re ^{n_u}$ are feedforward controls.
Problem \ref{prob:matrix_filtered_problem_affine} may thus be solved by identifying the sequences $(K_{k, i})$ and $(m_k)$.
However, before attempting to solve Problem \ref{prob:matrix_filtered_problem_affine}, we consider the following conservative relaxation of the chance constraints~(\ref{eq:chance_constraint}) given by the condition
\begin{equation} \label{eq:relaxed_chance_constraint}
	\sum_{j = 1}^{N_\chi} p_j \leq p_{\text{fail}}, \quad \Pr(\alpha_j \t x_k > \beta_j) \leq p_j, \;\; j = 1,\ldots,N_\chi.
\end{equation}
This constraint represents a decomposition of (\ref{eq:chance_constraint}) into independent half-plane constraints, which we expect to be a conservative approximation due to the subadditivity of probabilities.
Indeed, it has been shown in \cite{blackmore2009convex} that if (\ref{eq:relaxed_chance_constraint}) holds, then the chance constraint~(\ref{eq:chance_constraint}) is satisfied.

\begin{theorem} \label{thm:convex}
Given $p_j$ as in the  chance constraint relaxation (\ref{eq:relaxed_chance_constraint}),
Problem \ref{prob:matrix_filtered_problem_affine}  is convex.
\end{theorem}

\begin{proof}
	Let $M \isomorph (m_k)$ be column a vector defined as $U$, and let
	\begin{equation}
		K = \begin{bmatrix}
			K_{0,0} & 0 & 0 & \cdots & 0 \\
			K_{1,0} & K_{1,1} & \ddots & \cdots & 0  \\
			\vdots & \vdots & \ddots & 0 & 0 \\
			K_{N - 1, 0} & K_{N - 1, 1} & \cdots & K_{N - 1, N - 1} & 0
		\end{bmatrix}.
	\end{equation}
	We may then write the control process as the matrix equation
	\begin{equation}
		U = K (\hat{X} - \bar{X}) + M.
	\end{equation}
	The filtered state process (\ref{eq:filtered_process_vector_def}) is thus given by
	\begin{equation}
		\hat{X} = A \hat{x}_{0\m} + B K (\hat{X} - \bar{X}) + B M + L \tilde{Y}.
	\end{equation}
	Since $\tilde{Y}$ has zero mean, it follows that
	\begin{equation}
		\bar{X} = \E ( \hat{X} ) = A \bar{x}_0 + B M,
	\end{equation}
	and thus the terminal constraint $\E(x_N) = \bar{x}_f$ is written as
	\begin{equation} \label{eq:convex_final_mean}
		E_N \bar{X} = E_N (A \bar{x}_0 + B M) = \bar{x}_f,
	\end{equation}
	which is affine in $M$, and hence convex.
	Since $U - \E(U) = K(\hat{X} - \bar{X})$, we have
	\begin{equation}
		\hat{X} - \bar{X} = A (\hat{x}_{0\m} - \bar{x}_0) + B K (\hat{X} - \bar{X}) + L \tilde{Y},
	\end{equation}
	which, after solving for $\hat{X} - \bar{X}$, we rewrite as
	\begin{equation} \label{eq:state_deviation_with_inverse}
		\hat{X} - \bar{X} = (I - BK) \inv [ A (\hat{x}_{0\m} - \bar{x}_0) + L \tilde{Y} ].
	\end{equation}
	Since $K$ is block lower-triangular and $B$ is strictly block lower-triangular, the matrix $I - BK$ is invertible.
	Following~\cite{skaf2010design}, we define the new decision variable $F$ as
	\begin{equation}
		F = K (I - BK) \inv \in \Re ^{N n_u \times (N + 1) n_x}.
	\end{equation}
	It follows that $F$ is block lower-triangular and satisfies
	\begin{equation} \label{eq:K_F_equality}
		I + B F = (I - B K) \inv.
	\end{equation}
	Furthermore, $K$ is a function of $F$ given by
	\begin{equation}
		K = F (I + BF) \inv,
	\end{equation}
	and therefore we may optimize over $F$ in place of $K$ \cite{skaf2010design}.
	Substituting (\ref{eq:K_F_equality}) into (\ref{eq:state_deviation_with_inverse}), we obtain
	\begin{equation} \label{eq:state_deviation_wrt_F}
		\hat{X} - \bar{X} = (I + BF) [ A (\hat{x}_{0\m} - \bar{x}_0) + L \tilde{Y} ].
	\end{equation}
	By assumption, $\hat{x}_{0\m}$ is independent from both $\tilde{x}_{0\m}$ and $v_0$, and therefore by (\ref{eq:innovation_process_algebraic}) we have that $\hat{x}_{0\m}$ is independent from $\tilde{Y}$.
	It follows that the bracketed term in the right-hand side of (\ref{eq:state_deviation_wrt_F}) has covariance
	\begin{equation}
		S = \mathrm{Cov}[ A (\hat{x}_{0\m} - \bar{x}_0) + L \tilde{Y} ] = A \hat{P}_{0\m} A \t + L P_{\tilde{Y}} L \t,
	\end{equation}
	where, since steps of $(\tilde{y}_{k\m})$ are independent \cite{Astrom1970}, the covariance of $\tilde{Y}$ is the block-diagonal matrix
	\begin{equation}
		P_{\tilde{Y}} = \E (\tilde{Y} \tilde{Y} \t) = \mathrm{blkdiag}(P_{\tilde{y}_{0\m}}, \dots, P_{\tilde{y}_{N\m}} ),
	\end{equation}
	where $P_{\tilde{y}_{k\m}}$ as in (\ref{eq:innovation_covariance}). The covariance of the filtered process is thus
	\begin{equation} \label{eq:block_filtered_covariance}
		\hat{P} = \E[(\hat{X} - \bar{X}) (\hat{X} - \bar{X}) \t ]= (I + BF) S (I + BF) \t,
	\end{equation}
	and the covariance of the control is given by
	\begin{equation}
		P_U = \E[(U - \E(U)) (U - \E(U)) \t ] = F S F\t.
	\end{equation}
	We can then rewrite the objective (\ref{eq:objective_matrix_form}) as
	\begin{multline} \label{eq:convex_objective}
		\hat{J}(F, M) = (A \bar{x}_0 + B M) \t Q (A \bar{x}_0 + B M) + M \t R M \\
	+ \tr \{ [(I + BF) \t Q (I + BF) + F \t R F] S \},
	\end{multline}
	which is convex in $F$ and $M$, because $Q \geq 0$ and $R > 0$. 
	The terminal covariance constraint~(\ref{eq:problem2TerminalCovConst}) may be written as 
	\begin{equation}
		E_N (I + BF) S (I + BF) \t  E_N \t \leq P_f -  \tilde{P}_N,
	\end{equation}
	or, equivalently~\cite{Okamoto2018},
	\begin{equation} \label{eq:convex_final_covariance}
		\Vert S ^{1/2} (I + BF) \t  E_N \t (P_f -  \tilde{P}_N) ^{-1/2} \Vert - 1 \leq 0,
	\end{equation}
	where $S^{1/2}$ denotes a matrix satisfying $S = (S^{1/2}) \t S^{1/2}$.
	The matrix $(P_f - \tilde{P}_N)^{-1/2}$ exists since $P_f > \tilde{P}_N$.
	Next, we consider the conservative chance constraints (\ref{eq:relaxed_chance_constraint}).
	The probabilistic half-plane constraint in (\ref{eq:relaxed_chance_constraint}) has been shown in \cite{Okamoto2018b} to be equivalent to the constraint
	\begin{equation}
		\ncdf \inv ( 1 - p_j) \Vert P_k ^{1/2} \alpha_j \Vert + \alpha_j \t \bar{x}_k - \beta_j \leq 0,
	\end{equation}
	where $\ncdf\inv$ is the inverse of the cumulative normal distribution function and $P_k$ is the state covariance at time step $k$, which we can write as
	\begin{equation}
		P_k = E_k \hat{P} E_k \t +  \tilde{P}_k.
	\end{equation}
	In addition, $P_k^{1/2}$ satisfies $(P_k ^{1/2}) \t P_k ^{1/2} = P_k$ and is obtained by 
	\begin{equation}
		P_k ^{1/2} = \begin{bmatrix}
		S ^{1/2} (I + BF) \t E_k \t \\ \tilde{P}_k ^{1/2}
	\end{bmatrix}.
	\end{equation}
	Notice that, because each $p_j < 0.5$, it follows that $\ncdf \inv (1 - p_j) > 0$. 
	 Finally, substituting into the chance constraint, we obtain the second order cone constraint
	\begin{multline} \label{eq:convex_chance_constraint}
		\ncdf \inv ( 1 - p_j) \left\Vert \begin{bmatrix}
				S ^{1/2} (I + BF) \t E_k \t \\ \tilde{P}_k ^{1/2}
			\end{bmatrix} \alpha_j \right\Vert \\
			+ \alpha_j \t E_k (A \bar{x}_0 + BM) - \beta_j \leq 0, \quad j = 1, \dots, N_\chi.
	\end{multline}	
Given some constants $p_j$,  it follows that  the problem of minimizing the objective (\ref{eq:convex_objective}) with respect to the optimization parameters $F$ and $M$, 
 subject to the constraints (\ref{eq:convex_final_mean}), (\ref{eq:convex_final_covariance}), (\ref{eq:convex_chance_constraint}) is convex. 
\end{proof}

We remark that the convex constraints (\ref{eq:convex_final_covariance}) and (\ref{eq:convex_chance_constraint}) depend on the estimation error covariance $\tilde{P}_k$, which, in turn, depend on the measurement model (\ref{eq:observation_process}) and the filter design.
Therefore, in contrast to separation-based control \cite{Tse1971}, the optimal control that solves Problem~\ref{prob:matrix_filtered_problem_affine} cannot be found independently of the optimal filter, provided that the constraints (\ref{eq:convex_final_covariance}) and (\ref{eq:convex_chance_constraint}) are active.
This result is intuitive: If the state estimate is highly uncertain, then additional control effort may be required to steer further away from an obstacle.
Furthermore, the certainty equivalence property \cite{BrysonHo} does not hold when the constraints (\ref{eq:convex_final_covariance}) and (\ref{eq:convex_chance_constraint}) are active, since the intensity of the process noise is represented in the \mbox{matrix $S$}.

The controller (\ref{eq:control_structure}) uses feedback of both the current and the past values of the filtered state process at each step.
It follows that the computational complexity of the convex formulation of Problem \ref{prob:matrix_filtered_problem_affine} as given in Theorem \ref{thm:convex} scales with $\mathcal{O}(N^2 n_x n_u)$.
For problems with a large time horizon, one may restrict the matrix $F$ to be block diagonal; the resulting computational complexity scales by $\mathcal{O}(N n_x n_u)$ \cite{Okamoto2018}.
More generally, the matrix $F$ can be set to be block banded, which allows the designer to trade controller performance with computational complexity \cite{skaf2010design}.

\begin{figure}[]
	\includegraphics[trim={0.0in 0.05in 0.0in 0.05in}, clip]{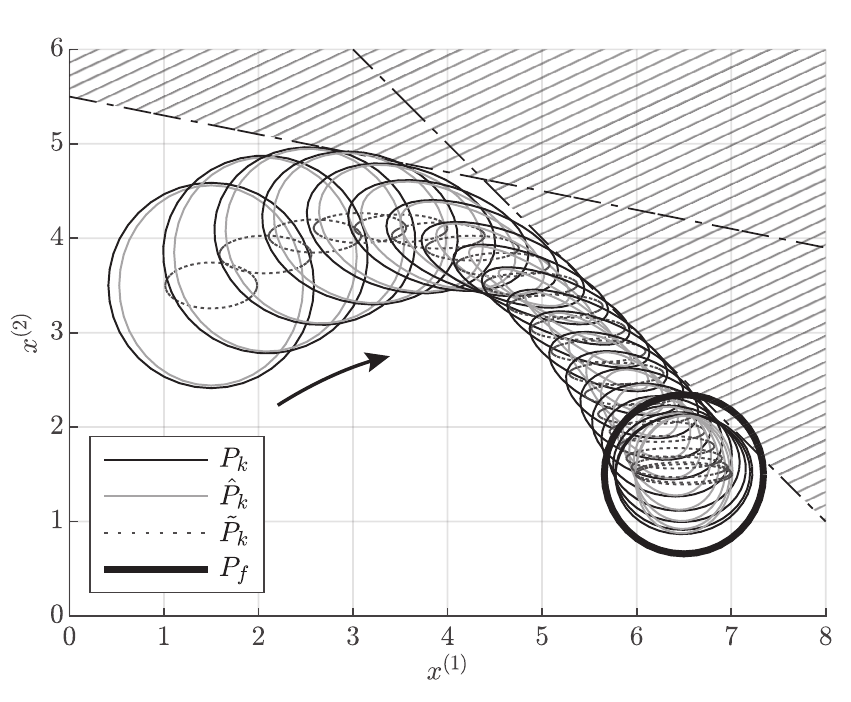}
	\includegraphics[trim={0.0in 0.08in 0.0in 0.2in}, clip]{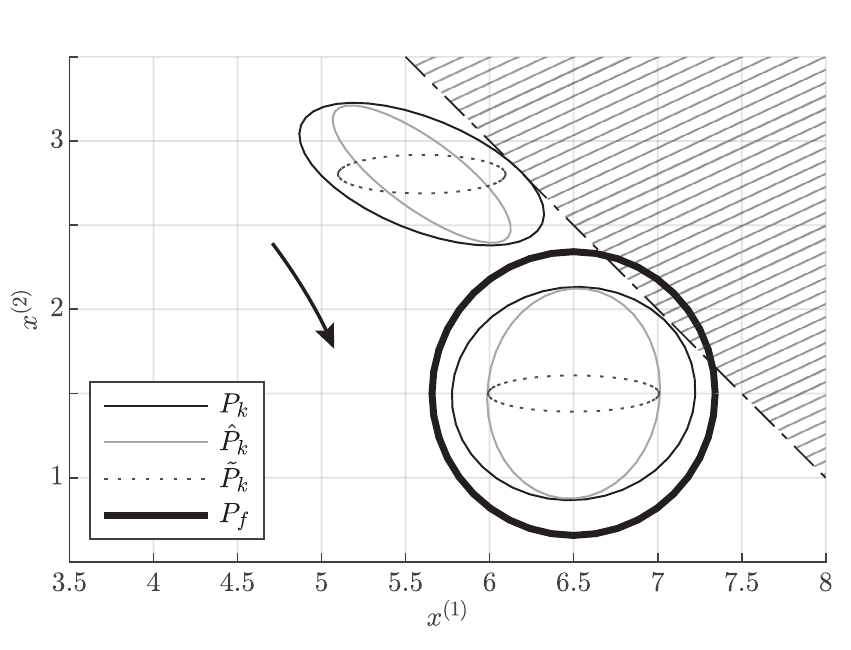}
	\caption{(Top:) 3$\sigma$ covariance ellipses drawn at each step; (Bottom:) detail of 3$\sigma$ covariance ellipses for steps $k = 11$ and $k = N = 20$ are shown. 
	The arrows indicate the direction of motion and the compliment of $\chi$ is marked by diagonal lines. 
	The variables $x_k^{(1)}$ and  $x_k^{(2)}$ denote the first two coordinates of the state. \label{fig:position_covariance}}
\end{figure}

\section{NUMERICAL EXAMPLE}\label{sec:NumericalExample}

Consider for $\Delta t = 0.2$ the following double integrator system with the horizon $N = 20$ given by, for all $k$,
\begin{equation}
	A_k = \begin{bmatrix}
		1 & 0 & \Delta t & 0 \\
		0 & 1 & 0 & \Delta t \\
		0 & 0 & 1 & 0 \\
		0 & 0 & 0 & 1
	\end{bmatrix}, \; B_k = \begin{bmatrix}
		 \Delta t^2 / 2 & 0 \\
		 0 & \Delta t^2 / 2 \\
		 \Delta t & 0 \\
		 0 & \Delta t
	\end{bmatrix}, 
\end{equation}
and $G_k = 0.01 \times I_3$, $C_k = \begin{bmatrix}
		0_{3\times 1} & I_3
	\end{bmatrix}$, $D_k = \diag(0.1, 0.003, 0.003)$.
and $G_k = 0.01 \times I_3$.
The initial state distribution is described by
$\tilde{P}_{0\m} = \mathrm{diag}(2, 1, 1.4, 1.4) \times 10^{-2}$, $\hat{P}_{0\m} = \mathrm{diag}(8, 9, 0.6, 0.6) \times 10^{-2}$, $\bar{x}_0 = [1.5, 3.5, 3, 2]\t$,
and the target state distribution is constrained to have mean
$\bar{x}_f = [6.5, 1.5, 0, 0]\t$ and maximum covariance $P_f = \mathrm{diag}(6, 6, 0.6, 0.6) \times 10^{-2}$.
Finally, the region $\chi$ is defined for $N_\chi = 2$ half plane constraints as in (\ref{eq:chance_constrained_region}) given by
$\alpha_1 = [1, 5, 0, 0]\t$,
$\alpha_2 =[1, 1, 0, 0]\t$,
with $\beta_1 = 27.5$, $\beta_2 = 9$, $p_1 = p_2 = 5 \times 10^{-4}$, and $p_{\text{fail}} = p_1 + p_2$.
We solved the convex optimization problem using YALMIP \cite{Lofberg2004yalmip} with MOSEK \cite{Mosek2015}.
The resulting trajectory of the distribution of the position coordinates are shown in \mbox{Figure \ref{fig:position_covariance}}.
%
%
In this example it is clear that the resulting control depends on the observation model.
Since the first position coordinate is not directly measured, there is a larger uncertainty in the estimated value of the first position coordinate compared to the second position coordinate.
The controller compensates accordingly by using sufficient control effort along the first position coordinate so that the chance constraints are satisfied.
We can see this by observing the shape of the $3 \sigma$ ellipse of the filtered state covariance $\hat{P}_k$ in the bottom plot of \mbox{Figure \ref{fig:position_covariance}}.

\section{CONCLUSIONS}\label{sec:Conclusion}

In this paper, we have developed a covariance steering control policy for discrete-time linear stochastic systems with Gaussian process and measurement noise.
The filtered state was obtained by a Kalman filter, and then, in terms of the filtered state, the covariance steering problem was posed as a deterministic convex optimization problem.
It was observed that, due to the covariance-based constraints on the state distribution, the resulting optimal control depends on both the process noise and measurement model. Finally, the developed theory was demonstrated using a numerical example.

In future work, the proposed approach can be extended to vehicle path planning problems under Gaussian disturbance and measurement uncertainty~\cite{patil2012estimating,agha2014firm}. In particular, the proposed approach can be extended to handle non-convex feasible regions as in~\cite{Okamoto2018b}, where the authors converted the original stochastic vehicle path planning problem to a deterministic mixed integer  convex programming problem.

\section*{ACKNOWLEDGMENT}

The work of the first author was supported by a NASA Space Technology Research Fellowship.
The work of the second and third authors was supported by NSF award CPS-1544814, by ARL under DCIST CRA W911NF-17-2-
0181, and by ONR award N00014-18-1-2828. 
The authors would also like to thank Dipankar Maity for many helpful discussions and suggestions.


\bibliography{cdc20}
\bibliographystyle{ieeetr}

\end{document}